\documentclass[3p,twocolumn]{elsarticle}

\usepackage[utf8]{inputenc} 
\usepackage[T1]{fontenc}    
\usepackage{subfigure}

\usepackage{amsmath,amssymb,amsfonts}
\usepackage{amsthm}
\usepackage{mathrsfs}
\usepackage{xcolor}
\usepackage{textcomp}
\usepackage{listings}
\usepackage{graphicx}
\usepackage{url}
\usepackage{color}
\usepackage{accents}
\usepackage[font=footnotesize]{caption}
\usepackage[font=footnotesize]{subcaption}
\usepackage{array}

\usepackage{algorithm}
\usepackage{algpseudocode}
\usepackage{bigints}

\newtheorem{theorem}{Theorem}[section]
\newtheorem{lemma}[theorem]{Lemma}

\newtheorem{proposition}[theorem]{Proposition}
\newtheorem{problem}{Problem}
\newtheorem{remark}[theorem]{Remark}
\newtheorem{assumption}{Assumption}
\newtheorem{example}[theorem]{Example}

\DeclareMathAlphabet{\mymathbb}{U}{BOONDOX-ds}{m}{n}

\newcommand{\bx}{{\mathbf{x}}}

\newcommand{\bB}{{\mathbf{B}}}

\newcommand{\bw}{{\mathbf{w}}}
\newcommand{\bu}{{\mathbf{u}}}
\newcommand{\bp}{{\mathbf{p}}}
\newcommand{\bq}{{\mathbf{q}}}

\newcommand{\bc}{{\mathbf{c}}}

\newcommand{\bk}{{\mathbf{k}}}

\newcommand{\setdef}[2]{\{#1 : #2\}}

\newcommand{\real}{\mathbb{R}}

\newcommand{\Sc}{\mathcal{S}}
\newcommand{\Pc}{\mathcal{P}}

\newcommand{\Bc}{\mathcal{B}}
\newcommand{\Cc}{\mathcal{C}}
\newcommand{\Dc}{\mathcal{D}}

\newcommand{\Fc}{\mathcal{F}}

\newcommand{\Kc}{\mathcal{K}}

\newcommand{\Tc}{\mathcal{T}}

\newcommand{\Nc}{\mathcal{N}}

\DeclareSymbolFont{bbold}{U}{bbold}{m}{n}
\DeclareSymbolFontAlphabet{\mathbbold}{bbold}

\newcommand{\norm}[1]{\lVert#1\rVert}
\newcommand{\normb}[1]{\big\lVert#1\big\rVert}

\newcommand\oprocendsymbol{\hbox{$\bullet$}}
\newcommand\oprocend{\relax\ifmmode\else\unskip\hfill\fi\oprocendsymbol}

\newcommand*{\QEDA}{\hfill\ensuremath{\blacksquare}}%

\newcommand\xqed[1]{%
  \leavevmode\unskip\penalty9999 \hbox{}\nobreak\hfill
  \quad\hbox{#1}}
\newcommand\demo{\xqed{$\bullet$}}

\newcommand{\longthmtitle}[1]{\mbox{}\emph{(#1):}}

\begin{document}

\begin{frontmatter}
    \title{\bf Universal Formulas for Safe Control and Their Neural Network Approximations}

    \author[1]{Pol Mestres\corref{cor1}}
    \ead{mestres@caltech.edu}
    \author[2]{Jorge Cort{\'e}s}
    \ead{cortes@ucsd.edu}
    \author[3]{Eduardo D. Sontag}
    \ead{e.sontag@northestern.edu}
    \cortext[cor1]{Corresponding author}
    \address[1]{Department of Mechanical and Civil Engineering, California Institute of Technology, Pasadena, CA, USA}
    \address[2]{Department of Mechanical and Aerospace Engineering, UC San Diego, La Jolla, CA, USA}
    \address[3]{Departments of Electrical and Computer Engineering and Bioengineering, Northeastern University, Boston, MA, USA}
    \begin{keyword}
        Universal formula, control Lyapunov function, control barrier function, neural network     
    \end{keyword}

    \begin{abstract}
        We study the problem of designing a controller that satisfies an arbitrary number of affine inequalities at every point in the state space.
        This is motivated by the fact that a variety of key control objectives, such as stability, safety, and input saturation, are guaranteed by closed-loop systems whose controllers satisfy such inequalities. 
        Many works in the literature design such controllers as the solution to a state-dependent quadratic program (QP) whose constraints are precisely the inequalities.
        When the input dimension and number of constraints are high, computing a solution of this QP in real time can become computationally burdensome.
        Additionally, the solution of such optimization problems is not smooth in general, which can degrade the performance of the system.
        This paper provides a novel method to design a smooth controller that satisfies an arbitrary number of affine constraints. 
        The controller is given at every state as the minimizer of a strictly convex function.
        To avoid computing the minimizer of such function in real time, we introduce a method based on neural networks (NN) to approximate the controller.
        Remarkably, this NN can be used to solve the controller design problem for any task with less than a fixed input dimension and number of affine constraints, and is completely independent of the state dimension.
        This is why we refer to such NN approximation as a NN-based universal formula for control.
        Additionally, we show that the NN-based controller only needs to be trained with datapoints from a bounded set in the state space, which significantly simplifies the training process.
        Various simulations showcase the performance of the proposed solution, and
        also show that the NN-based controller can be used to warmstart an optimization scheme that refines the approximation of the true controller in real time, significantly reducing the computational cost compared to a generic initialization.
    \end{abstract}
\end{frontmatter}

\section{Introduction}

Modern autonomous systems, ranging from self-driving vehicles to aerospace systems, are usually subject to a variety of operational requirements.
These include, but are not limited to, 
trajectory tracking, 
disturbance rejection, stabilization, and guaranteeing state and input constraints.
Although there exist technical tools to design controllers that satisfy each of these specifications, combining them into a single controller that meets all of them is often challenging.
For example, control Lyapunov (respectively, barrier) functions define state-dependent input constraints that guarantee stability (respectively, safety)
of the underlying control system.
However, designing a smooth control law that satisfies all such constraints and can be computed efficiently in real time is significantly difficult.
Addressing this problem is the main motivation for our work here.

\subsection{Literature Review}
Control Lyapunov Functions (CLFs)~\citep{EDS:98,EDS:83} provide a powerful tool to achieve the stabilization of nonlinear control systems.
For control-affine systems, CLFs prescribe an affine inequality in the system input at every state. By design, controllers that satisfy this inequality, such as the pointwise minimum norm controller~\citep{RAF-PVK:96a}
or the CLF universal formula~\citep{EDS:89a}, ensure stability of the closed-loop system.
More recently, Control Barrier Functions (CBFs)~\citep{ADA-SC-ME-GN-KS-PT:19,ADA-XX-JWG-PT:17,PW-FA:07} have been introduced to achieve safety requirements for nonlinear control systems.
Similarly to CLFs, CBFs also prescribe affine inequalities in the input at every state for control-affine systems, and various safe control designs leveraging this inequality have been proposed~\citep{ADA-SC-ME-GN-KS-PT:19,MC-PO-GB-ADA:23,ML-ZS-PJWK-SW:24}.
Affine inequalities in the input
also arise if control authority is limited (as is always the case in practice), or if stability 
needs to be guaranteed in fixed time~\citep{KG-DP:21}.
Often, it is common to face scenarios where the controller must meet multiple objectives, which are encoded by multiple affine inequalities in the input.
There are also a variety of works in the literature that design controllers in this setting.
For example, the CLF-CBF QP in~\citep{ADA-XX-JWG-PT:17,ML-ZS:23-acc} includes both affine constraints in a quadratic program (QP), whose solution defines the desired controller.
Alternatively,~\citep{MZR-BJ:16} unites a CLF and a CBF into a unique function, called control Lyapunov barrier function (CLBF), and uses the known CLF universal formula to derive a controller.
Other control designs include~\citep{PM-JC:23-csl}, which use penalty methods to impose one of the inequalities as a \textit{hard} constraint and the other one as a \textit{soft} constraint,~\citep{PO-JC:19-cdc}, which defines a formula based on the centroids of the two feasible sets defined by the two inequalities, or~\cite{TGM-ADA:23}, where different CBF constraints are united into a single one by designing a smooth approximation of the intersection of the different safe sets. 

However, the aforementioned approaches only lead to a controller that can be obtained in closed-form or in a computationally tractable manner for a low number of constraints~\citep{XT-DVD:24,PM-SSM-PO-LY-ED-JWB-ADA:25}.
To the best of our knowledge, there does not exist an approach in the literature for obtaining a controller with such properties for an arbitrary number of state-dependent affine constraints.
This problem raises intertwined theoretical, practical, and computational challenges. Theoretically, the solution of the QP should be computed at every state, which is impossible in practice. This has led to the use of sampled-data implementations of such controllers~\citep{AS-YC-ADA:20,AJT-PO-JC-AA:21-csl,JB-KG-DP:22}. Even with a sampled-data implementation, from a computational standpoint, solving the QP in real time becomes burdensome and the mismatch between the computed solution and the actual one may degrade system performance.

The problem of designing such a controller is even more challenging if it is required to be smooth. 
The use of smooth controllers is motivated by theoretical reasons (e.g., existence and uniqueness of solutions and use in backstepping designs~\citep{AJT-PO-TGM-ADA:22}) and practical considerations (e.g., avoidance of chattering behavior in digital platforms). Although the existence of an infinitely-times differentiable controller satisfying an arbitrary number of affine state-dependent constraints is guaranteed by~\citep[Proposition 3.1]{PO-JC:19-cdc}, a constructive method for finding such a
controller is lacking. In fact, optimization-based designs such as the QP in~\citep{ADA-XX-JWG-PT:17} can fail to be locally Lipschitz (as exemplified by Robinson's counterexample~\citep{BJM-MJP-ADA:15}) or even continuous~\citep{MA-NA-JC:25-tac}, and additional constraint qualifications are needed in order to ensure their regularity~\citep{BJM-MJP-ADA:15,PM-AA-JC:25-ejc}.

Our work here is also related to a line of research that uses neural networks (NNs) to approximate the solution of model predictive control (MPC) problems~\citep{SWC-TW-NA-VK-MM:22,SC-KS-NA-DDL-VK-GP:18,MH-JK-ST-FA:18,XZ-MB-FB:19}.
In the case of linear dynamics and affine input and state constraints,
the solution of such MPC problem is known~\citep{FB-AB-MM:17} to be piecewise affine, but computing it explicitly is difficult because the number of regions grows exponentially with the number of constraints.
These works show that an approximate solution computed offline with a NN can be used to warmstart the MPC optimization problem online and considerably speed up its computation.

\subsubsection*{Statement of Contributions}

We study the problem of designing a smooth controller that satisfies an arbitrary number of affine constraints at every point in the state space.
In our first contribution 
we introduce a novel controller that 
generalizes the CLF universal formula and
satisfies all the affine inequalities.
Such controller is different from other existing CLF/CBF-based controllers in the literature.
We show that such controller is smooth, and can be computed as the minimizer of a strictly convex function which has not been introduced elsewhere in the literature.

Since a closed-form expression for the controller is not available for cases with more than one constraint, our second contribution consists of a numerical scheme to approximate the controller using NNs.
Although similar schemes have been introduced elsewhere in the literature to approximate other controllers,
we show that in our case, remarkably, a single neural network (appropriately trained with parameters equal to the coefficients of the affine inequalities) is enough to obtain controllers for any task involving systems with the same input dimension and the same number of inequalities, regardless of the state dimension. Additionally, we show that such NN approximation can be obtained by only sampling data points over a compact set of parameters, and that the NN-based approximate controller is smooth (by taking the activation functions of the NN to be smooth). Although such approximation is not guaranteed to be close to the true controller and to satisfy the inequalities at every point, the universal approximation theorem of NNs ensures that it becomes an arbitrarily good approximation of the true controller as the number of parameters of the NN increases.
Additionally, using recent advances in training NNs with hard constraints (such as HardNet~\cite{YM-NA:25} or $\Pi$-net~\cite{PDG-AT-ECB-RD-JL:25}), the NN can be trained so that the NN-based controller satisfies the affine constraints at every point by construction, although in that case it is not guaranteed to be smooth.
Furthermore, in real-time control applications, the NN approximations can be used to warmstart an optimization scheme to compute the true controller with higher accuracy.

Finally, in our third contribution we use the NN approximations in two safe stabilization tasks, both directly and as a warmstart of an optimization scheme to compute the true controller value.
We show that the use of the NN significantly reduces the execution time of the controller compared to other standard controllers 
in the literature such as the minimum-norm CLF-CBF QP controller~\citep{ADA-XX-JWG-PT:17}, which need to be computed online by solving an optimization problem at every state.

\subsubsection*{Notation}
We denote by $\mathbb{Z}_{>0}$ and $\real$ the set of positive integers and real numbers, respectively.
For $N\in\mathbb{Z}_{>0}$, we write 
$[N] = \{ 1, \hdots, N \}$.
Given $l\in\mathbb{Z}_{>0}$ and $\Sc\subset\real^n$,
the set of $l$-times continuously differentiable functions in $\Sc$ is denoted by $\Cc^l(\Sc)$.
Vectors are represented by boldface symbols whereas scalars are represented by non-boldface symbols.
The zero vector in $\real^n$ is denoted by $\mathbf{0}_n$, and the identity matrix of dimension $n\times n$, by $\mathbf{I}_n$.
Given $\bx\in\real^n$, $\norm{\bx}$ denotes its Euclidean norm.
A function $\beta:\real\to\real$ is of class $\Kc$ if $\beta(0)=0$ and $\beta$ is strictly increasing. If moreover, $\lim\limits_{t\to\infty}\beta(t)=\infty$ and $\lim\limits_{t\to-\infty}\beta(t)=-\infty$, then $\beta$ is of extended class $\Kc_{\infty}$.
A function $V:\real^n\to\real$ is positive definite if $V(\mathbf{0}_n) = \mathbf{0}_n$ and $V(\bx) > 0$ for all $\bx\neq\mathbf{0}_n$.
The ball centered at $\bp\in\real^n$ with radius $r > 0$ is denoted by $\Bc_r(\bp)$.
Let $F:\real^n\to\real^n$ be a locally Lipschitz vector field and consider the dynamical system $\dot{\bx} = F(\bx)$.
Local Lipschitzness of $F$ ensures that, for every initial condition $\bx_0\in\real^n$, there exists $T>0$ and a unique trajectory $x:[0,T]\to\real^n$ such that $x(0) = \bx_0$ and $\dot{x}(t)=F(x(t))$.
If all solutions exist for all $t\geq0$, then we say the dynamical system is forward complete.
In this case, we let
$\Phi_t:\real^n\to\real^n$ denote the flow map, defined by $\Phi_t(\bx) = x(t)$, where $x(t)$ is the unique solution of the dynamical system starting at $x(0) = \bx$. 
Given a forward complete dynamical system, a set $\Kc\subset\real^n$ is (positively) forward invariant if $\bx\in\Kc$ implies that $\Phi_t(\bx)\in\Kc$ for all $t\geq0$.
A point $\bp$ is Lyapunov stable if, for every open set $U$ containing $\bp$, there exists an open set $\tilde{U}$ also containing $\bp$ such that for all $\bx\in\tilde{U}$, $\Phi_t(\bx)\in U$ for all $t\geq0$.
An equilibrium point $\bp$ is asymptotically stable if it is Lyapunov stable and there exists an open set $\bar{U}$ containing $\bp$ such that $\Phi_t(\bx)\to \bp$ as $t\to\infty$ for all $\bx\in\bar{U}$.

\section{Problem Statement}\label{sec:problem-statement}

Consider a nonlinear control-affine system of the form
\begin{align}\label{eq:control-affine-system}
    \dot{\bx} = f(\bx) + g(\bx)\bu,
\end{align}
where $f:\real^n\to\real^n$ and 
$g:\real^n\to\real^{n\times m}$
are locally Lipschitz, with $\bx\in\real^n$ the state and $\bu\in\real^m$ the input.
We assume that $f(\mathbf{0}_n) = \mathbf{0}_n$, so that the origin is an equilibrium of the unforced system.
The following motivating examples provide instances on how state-dependent affine inequalities in the input arise when designing controllers aimed at meeting specific control objectives.

\begin{example}\longthmtitle{Control Lyapunov function for stabilization}\label{ex:clf}
    A continuously differentiable positive definite function $V:\real^n\to\real$ with compact level sets is a control Lyapunov function (CLF) if there exists a neighborhood $\Dc$ of the origin such that for all $\bx\in\Dc$, there exists $\bu\in\real^m$ such that 
    \begin{align}\label{eq:clf-ineq}
        \nabla V(\bx)^\top f(\bx) + \nabla V(\bx)^\top g(\bx) \bu + W(\bx) \leq 0,
    \end{align}
    where $W:\real^n\to\real$ is a positive definite function.
    A locally Lipschitz controller $k:\real^n\to\real^m$ such that $\bu = k(\bx)$ satisfies the inequality~\eqref{eq:clf-ineq} at every $\bx\in\Dc$ renders the origin of the closed-loop system $\dot{\bx} = f(\bx) + g(\bx)k(\bx)$ asymptotically stable.
    \demo
\end{example}

\begin{example}\longthmtitle{Control barrier function for ensuring safety}\label{ex:cbf}
    Let $\Cc\subset\real^n$ be a safe set of interest, and suppose that $\Cc$ is given by the $0$-superlevel set of a continuously differentiable function $h:\real^n\to\real$, i.e., $\Cc = \setdef{\bx\in\real^n}{h(\bx) \geq 0}$.
    This function is a control barrier function (CBF) of $\Cc$ if, for every $\bx\in\Cc$, there exists $\bu\in\real^m$ such that 
    \begin{align}\label{eq:cbf-ineq}
        \nabla h(\bx)^\top f(\bx) + \nabla h(\bx)^\top g(\bx) \bu + \alpha(h(\bx)) \geq 0,
    \end{align}
    where $\alpha:\real\to\real$ is an extended class $\Kc_{\infty}$ function.
    A locally Lipschitz controller $k:\real^n\to\real^m$ such that $\bu = k(\bx)$ satisfies the inequality~\eqref{eq:cbf-ineq} at every $\bx\in\Cc$ renders the set $\Cc$ forward invariant for the closed-loop system $\dot{\bx} = f(\bx) + g(\bx) k(\bx)$.
    \demo
\end{example}

Examples~\ref{ex:clf} and~\ref{ex:cbf} provide two illustrations of control-theoretic properties (stability and safety) that can be achieved by designing controllers that  satisfy state-dependent affine inequalities in the input. In general, one can have an arbitrary number of such affine inequalities to encode, for instance, safety constraints related to collision avoidance
with different obstacles in the environment.

Formally, let $N\in\mathbb{Z}_{>0}$, and for $i\in[N]$, let  $a_i:\real^n\to\real$, $b_i:\real^n\to\real^{m}$ be functions for some $l\in\mathbb{Z}_{\geq 0}$. Let $\Fc$ be the set of points where the inequalities 
$\{a_i(\bx) + b_i(\bx)^\top \bu < 0\}_{i \in [N]}$
are simultaneously strictly feasible, i.e., 
\begin{align*}
    \Fc &:= \setdef{ \bx\in\real^n }{ \exists \bu\in\real^m \ \ \text{s.t.} \\
    &\quad \quad \ a_i(\bx) + b_i(\bx)^\top \bu < 0, \; \forall i\in[N] }.
\end{align*}

Throughout the paper, we make the following assumptions:
\begin{assumption}\longthmtitle{Regularity}\label{as:regularity}
    The functions $\{ a_i, b_i \}_{i=1}^N$ belong to $\Cc^{l}(\real^n)$ for some $l\in\mathbb{Z}_{\geq0}$.
\end{assumption}

\begin{assumption}\longthmtitle{Feasible set}\label{as:feasible-set}
    The set $\Fc$ is non-empty.
\end{assumption}

Since $\Fc$ is the projection in $\real^n$ 
of the open set $\setdef{(\bx,\bu)\in\real^n\times\real^m}{a_i(\bx)+b_i(\bx)^\top \bu < 0, \ \forall i\in[N]}$, $\Fc$ is also open. 
In the case where $N=2$ and the two inequalities correspond to the strict versions of the CLF and CBF inequalities~\eqref{eq:clf-ineq} and~\eqref{eq:cbf-ineq}, respectively, 
we have
\begin{alignat*}{2}
    a_1(\bx) &= \nabla V(\bx)^\top f(\bx) + W(\bx), \; \\ 
    b_1(\bx) &= g(\bx)^\top \nabla V(\bx),
    \\
    a_2(\bx) &= -\nabla h(\bx)^\top f(\bx) - \alpha(h(\bx)) , \; \\
    b_2(\bx) &= -g(\bx)^\top \nabla h(\bx).
\end{alignat*}
In this case, the property that $\Fc\neq\emptyset$ (cf., Assumption~\ref{as:feasible-set}) is known as strict compatibility of the CLF and CBF pair, cf.~\citep{PO-JC:19-cdc,PM-JC:25-tac}.
\cite[Proposition 3.1]{PO-JC:19-cdc} is an extension of Artstein's Theorem (cf.~\citep{ZA:83}) showing that if 
$\Fc$ is non-empty,
there exists a $\Cc^{\infty}$ controller $k:\Fc\to\real^m$ such that 
$a_i(\bx) + b_i(\bx)^\top k(\bx) < 0$ 
for all $\bx\in\Fc$ and all $i \in [N]$.
The proof of this result, however, is not constructive, which is undesirable for the real-time control of the nonlinear system~\eqref{eq:control-affine-system}. Hence, in this paper we set out to solve the following problem.

\begin{problem}\label{problem:problem}
    Design a smooth controller that satisfies the constraints 
    $a_i(\bx) + b_i(\bx)^\top \bu < 0$ 
    for all $i\in[N]$ and all $\bx\in\Fc$.
\end{problem}

\section{A Universal Formula for an Arbitrary Number of Affine Constraints}\label{sec:a-universal-formula-for-an-arbitrary-number-of-affine-constraints}

In this section we introduce a smooth universal formula for a controller satisfying an arbitrary number of affine inequalities.
Let $A_1, \hdots, A_N \in\real$ and $\bB_1, \hdots, \bB_N \in \real^m$. Let 
$\bp = (A_1, \hdots, A_N, \bB_1, \hdots, \bB_N)$ and define
$\Kc_{\bp} := \setdef{ \bk\in\real^m }{ A_i + \bB_i^\top \bk < 0, \; \forall i\in[N] }$.
Note that $\Kc_{\bp}$ is a convex polytope.
Now, define the function $J_{\bp}:\Kc_{\bp}\to\real$ as 
\begin{align}\label{eq:Jp-definition}
    J_{\bp}(\bk) = -\sum_{i=1}^N \frac{\norm{\bB_i}^2+\norm{\bk}^2}{2(A_i + \bB_i^\top \bk  )}.
\end{align}

Our first result shows that $J_{\bp}$ is strictly convex.

\begin{proposition}\longthmtitle{Strict convexity}\label{prop:strict-convexity}
    Let $\bp \in \real^{N+mN}$.
    The function $J_{\bp}$ is strictly convex in the convex domain $\Kc_{\bp}$.
\end{proposition}
\begin{proof}
Let $J_{i,\bp}:\Pc\to\real$ be defined as
\begin{align*}
    J_{i,\bp}(\bk) = -\frac{\norm{\bB_i}^2+\norm{\bk}^2}{2( A_i + \bB_i^\top \bk )}.
\end{align*}
Note that the Hessian of $J_{i,\bp}$ is given by
\begin{align*}
    &\nabla^2 J_{i,\bp}(\bk) = \frac{-\Gamma(\bk)}{ (A_i \!+\! \bB_i^\top \bk)^3 },
\end{align*}
where 
$\Gamma(\bk) = (A_i + \bB_i^\top \bk)^2 \mathbf{I}_n - ( \bk \bB_i^\top + \bB_i \bk^\top )( A_i + \bB_i^\top \bk) + ( \norm{\bB_i}^2 +\norm{\bk}^2)\bB_i \bB_i^\top$. 

$\bullet$ \underline{Proving that $\Gamma(\bk)$ is positive definite:}

Let us show that
$\bx^\top \Gamma(\bk) \bx > 0$ whenever $\bx \in \real^m \backslash \{ \mathbf{0}_m \}$.
This is true if and only if 
\begin{align*}
    &(A_i + \bB_i^\top \bk)^2 \norm{\bx}^2 - 2(A_i + \bB_i^\top \bk)(\bB_i^\top \bk)(\bk^\top \bx) + \\
    &(\norm{\bB}_i^2 + \norm{\bk}^2)(\bB_i^\top \bx)^2 > 0,
\end{align*}
for all $\bx\in\real^m\backslash \{ \mathbf{0}_m \}$.
Let $z = \bB_i^\top \bx$. We want to show that the scalar quadratic function in $z$ given by $(A_i + \bB_i^\top \bk)^2 \norm{\bx}^2 - 2(A_i + \bB_i^\top \bk) (\bk^\top \bx) z + (1+\norm{\bk}^2)z^2$ is strictly positive for all $\bx\in\real^m\backslash\{\mathbf{0}_m\}$. Note that the discriminant of the quadratic is 
\begin{align*}
    \Delta \! = \! 4(A_i \! + \! \bB_i^\top \bk)^2 \Big( 
    (\bk^\top \bx)^2 \! - \! \norm{\bk}^2 \norm{\bx}^2 \! - \! \norm{\bx}^2 \norm{\bB_i}^2 \Big)
\end{align*}
By the Cauchy-Schwartz inequality, $(\bk^\top \bx)^2 \leq \norm{\bk}^2 \norm{\bx}^2$. Therefore, 
$(\bk^\top \bx)^2 - \norm{\bk}^2 \norm{\bx}^2 - \norm{\bk}^2 \norm{\bB_i}^2 < 0$ unless $\bB_i = \mathbf{0}_m$ (in which case $(A_i + \bB_i^\top \bk)^2 \norm{\bx}^2 - 2(A_i + \bB_i^\top \bk) (\bk^\top \bx) z + (1+\norm{\bk}^2)z^2 > 0$ for $\bx\neq\mathbf{0}_n$) or $\bx = \mathbf{0}_m$.
This implies that $\Delta < 0$ unless $\bx = \mathbf{0}_m$ and therefore the quadratic $(A_i + \bB_i^\top \bk)^2 \norm{\bx}^2 - 2(A_i + \bB_i^\top \bk) (\bk^\top \bx) z + (1+\norm{\bk}^2)z^2$ is strictly positive for all $\bx\in\real^m\backslash\{ \mathbf{0}_m \}$.

$\bullet$ \underline{Proving that $J_{\bp}$ is strictly convex:}

Since $\Gamma$ is positive definite, this means that $\nabla^2 J_{i,\bp}(\bk)$ is positive definite
for all $\bk\in\Kc$, and $J_{i,\bp}$ is strictly convex in $\Kc$. Since $J_{\bp}(\bk) = \sum_{i=1}^N J_{i,\bp}(\bk)$, $J_{\bp}$ is the sum of strictly convex functions and is therefore also strictly convex in $\Kc_{\bp}$.
\end{proof}

Since $J_{\bp}$ is strictly convex in $\Kc_{\bp}$ by Proposition~\ref{prop:strict-convexity}, and $J_{\bp}$ goes to infinity as it approaches the boundary of $\Kc_{\bp}$, it follows that $J_{\bp}$
has a unique minimizer in $\Kc_{\bp}$. Let 
\begin{align*}
    \Pc = \setdef{ \bp\in\real^{N+mN} }{ \Kc_{\bp} \neq \emptyset }.
\end{align*}
Then, we let
$k^*:\Pc \to \real^m$ be the function mapping every tuple 
$\bp\in\Pc$ to the unique minimizer of $J_{\bp}$ in $\Kc_{\bp}$. 
The following result shows that the mapping $k^*$ is smooth.

\begin{lemma}\longthmtitle{Smoothness}\label{lem:smoothness}
    The function $k^*$ is $\Cc^{\infty}(\Pc)$.
\end{lemma}
\begin{proof}
    Since $\bp\in\Pc$, $\Kc_{\bp}\neq\emptyset$.
    First note that since $\lim\limits_{\norm{\bk}\to\infty}J_{\bp}(\bk) = \infty$, $k^*(\bp)$ is finite for all $\bp$.
    Since $J_{\bp}$ is strictly convex, continuous in $\Kc_{\bp}$ and $k^*(\bp)$ is its unique minimizer, 
    there exist neighborhoods $\Nc_1$, $\Nc_2$ of $k^*(\bp)$ and a constant $M>0$ such that
    if $\bk\in\Nc_1$ then $J_{\bp}(\bk) < M$, and 
    if $\bk\notin\Nc_2$ then $J_{\bp}(\bk) > 2M$.
    Now, since $J_{\bp}$ is continuous with respect to $\bp$, there exists a neighborhood $\Nc_{\bp}$ of $\bp$ such that for any $\bar{\bp}\in\Nc_{\bp}$, we have that if $\bk\in\Nc_1$ then $J_{\bar{\bp}}(\bk) < \frac{3M}{2}$, and if $\bk\notin\Nc_2$ then $J_{\bp}(\bk) > \frac{3M}{2}$.
    Hence, 
    $k^*(\bar{\bp})\in\Nc_1$ for all $\bar{\bp}\in\Nc_{\bp}$.
    Now, define the function 
    $J:\Nc_1\times\Nc_1 \to \real$ as $J(\bp,\bk) = J_{\bp}(\bk)$.
    Note that since $k^*$ is the minimizer of $J_{\bp}$, it satisfies the equation 
    $\frac{\partial J}{\partial \bk}(\bar{\bp},k^*(\bar{\bp})) = \mathbf{0}_m$ for all $\bar{\bp}\in\Nc_1$.
    Additionally, note that since $J_{\bp}$ is strictly convex as shown in Proposition~\ref{prop:strict-convexity}, $\frac{\partial^2 J}{\partial^2 \bk}(\bar{\bp},k^*(\bar{\bp}))$ is non-singular.
    Now, by the Implicit Function Theorem~\citep[Proposition 1B.5]{ALD-RTR:14}, since $J$ is $\Cc^{\infty}$, $k^*$ is also $\Cc^{\infty}$ at $\bp$. This argument is valid for any $\bp\in\Pc$.
\end{proof}

Therefore, $k^*$ is $\Cc^{\infty}$ and satisfies the constraints $A_i + \bB_i^\top k^* (\bp) < 0$ for all $i\in[N]$ and $\bp \in \Pc$. Define now the controller 
$u^*:\Fc\to\real^m$
by
\begin{align}\label{eq:controller}
 u^*( \bx ) = k^*( a_1(\bx), \hdots, a_N(\bx), b_1(\bx), \hdots, b_N(\bx) ) .
\end{align}
Note that $\bx \in \Fc$ implies that 
$( a_1(\bx), \hdots, a_N(\bx), b_1(\bx), \hdots, b_N(\bx) ) \in \Pc$. 
Therefore, 
$u^*$ is $\Cc^{l}(\Fc)$
and satisfies the constraints $a_i(\bx) + b_i(\bx)^\top u^*(\bx)$ for all $i\in[N]$ and $\bx \in \Fc$, solving Problem~\ref{problem:problem}.

\begin{remark}\longthmtitle{Connection with CLF universal formula}\label{rem:Sontag-formula}
    A noteworthy property of the controller $u^*$ is that when $N=1$, $m=1$, 
    and the constraint corresponds to a CLF constraint that is satisfied strictly, i.e., with $a_1(\bx) = \nabla V(\bx)^\top f(\bx) + W(\bx)$ and $b_1(\bx) = \nabla V(\bx)^\top g(\bx)$ for some CLF $V$ and positive definite function $W$,
    then $k^*$ equals the well-known universal formula~\citep{EDS:89a} for stabilization.
    Indeed, the minimizer of $J_{\bp}$ can be found by solving the nonlinear equation $\nabla J_{\bp}(k) = -\frac{k}{A_1 + B_1 k} + \frac{B_1^2 + k^2 }{ 2(A_1 + B_1 k)^2 }B_1 = 0$.   By solving the resulting quadratic equation, one obtains the CLF universal formula:
    \begin{align*}
        k^*(A_1,B_1) = \begin{cases}
            -\frac{A_1 + \sqrt{A_1^2 + B_1^4} }{B_1}, \ &\text{if} \ B_1 \neq 0, \\
            0, &\text{otherwise}. \qquad \quad \demo
        \end{cases}   
    \end{align*}
\end{remark}

\begin{remark}\longthmtitle{Dynamical controller}\label{rem:dynamical-controller}
    Instead of computing the minimizer $u^*$ at every state, one can choose to run the following dynamical controller:
    \begin{subequations}
    \begin{align}
        \dot{\bx} &= f(\bx) + g(\bx) \bu, \\
        \dot{\bu} &= -\tau \frac{\partial J}{\partial \bu}(\bx,\bu),
    \end{align}
    \label{eq:interconnection}
    \end{subequations}
    where $J(\bx,\bu) = J_{\bx}(\bu)$ and $\tau>0$ is a design parameter.
    Then, by using singular perturbation theory~\cite[Chapter 11]{HK:02}, one can show that for sufficiently large $\tau$, the evolution of the state variable $\bx$ according to~\eqref{eq:interconnection} can be made arbitrarily close to the evolution of the state variable for $\dot{\bx} = f(\bx) + g(\bx)u^*(\bx)$.
    \demo
\end{remark}

\begin{remark}\longthmtitle{Alternative universal formula}\label{rem:alternative-universal-formula}
    Given a vector of positive weights $\bw\in\real^N$ (i.e., with components $w_i > 0$ for all $i\in[N]$), one can instead consider a variation of the function $J_{\bp}$ where each of the summands is assigned different weights:
    \begin{align*}
        J_{\bp}^{\bw}(\bk) = -\sum_{i=1}^N w_i \frac{ \norm{\bB_i}^2 + \norm{\bk}^2 }{2(A_i + \bB_i^\top \bk)}.
    \end{align*}
    By following an argument analogous to the one in the proof of Proposition~\ref{prop:strict-convexity}, it follows that $J_{\bp}^{\bw}$ is strictly convex in $\Kc_{\bp}$ and therefore it has a unique minimizer in $\Kc_{\bp}$, which defines an alternative smooth function satisfying all the constraints.
    \demo
\end{remark}

The following result shows that for safe stabilization problems, even though $u^*$ is not defined at the origin, trajectories in its neighborhood converge to it.

\begin{proposition}\longthmtitle{Convergence to the origin}
    Let $V:\real^n\to\real$ be a CLF, $h:\real^n\to\real$ be a CBF of a safe set $\Cc\subset\real^n$. Let $N=2$ and define $a_1(\bx) = \nabla V(\bx)^\top f(\bx) + W(\bx)$, $b_1(\bx) = g(\bx)^\top \nabla V(\bx)$,
    $a_2(\bx) = -\nabla h(\bx)^\top f(\bx) - \alpha h(\bx)$, $b_2(\bx) = -g(\bx)^\top \nabla h(\bx)$.
    Suppose that there exists $\bar{k} > 0$ such that 
    \begin{align*}
        \mathcal{O}_{\bar{k}} = \setdef{\bx\in\real^n}{0 < V(\bx) < \bar{k} } \subset \Fc.
    \end{align*}
    Let $\bx_0\in\mathcal{O}_{\bar{k}}$ with $V(\bx_0) = k \leq \bar{k}$ and consider the maximal solution of $\dot{\bx} = f(\bx)+g(\bx)u^*(\bx)$ in $\Fc$ with initial condition at $\bx_0$. 
    We let $\bx(\cdot;\bx_0)$ be such solution and $[0,T)$ its interval of existence (with either $T = \infty$ or $T < \infty$).
    Then, we have
    \begin{align*}
        \lim\limits_{t\to T} \bx(t;\bx_0) = \mathbf{0}_n.
    \end{align*}
\end{proposition}
\begin{proof}
    Suppose first that there exists a compact subset $C \subset \Fc$ such that $x(t)\in C$ for all $t\in[0,T)$.
    Then, by~\cite[Proposition C.3.6]{EDS:98}, it follows that $T=\infty$.
    Since $\bx(\cdot;\bx_0)$ is precompact, the omega limit set of $\bx_0$, denoted as $\omega^{+}(\bx_0)$ is non-empty.
    Since $V$ is a LaSalle function in $\mathcal{O}_{\bar{k}}$, for all $\boldsymbol{\xi}\in\omega^{+}(\bx_0)$ we have 
    $\dot{V}(\boldsymbol{\xi}) = 0$.
    However, since $\boldsymbol{\xi}\in\omega^{+}(\bx_0)\in\mathcal{O}_{\bar{k}}$, we have $\dot{V}(\boldsymbol{\xi}) < 0$, reaching a contradiction.
    Therefore, there exists no such set $C$.
    Let $\epsilon > 0$ and consider the set $C_{\epsilon,k} = \setdef{\bx\in\real^n}
    {\epsilon \leq V(\bx) \leq k}$.
    Since $C_{\epsilon,k}$ is a compact set contained in $\Fc$, there exists $t_1 \in (0, T)$ such that $\bx(t_1;\bx_0)\notin C_{\epsilon,k}$.
    Let $t_0 = \max_s \setdef{s \leq t_1 }{ \bx(s;\bx_0)\in C_{\epsilon,k} \ \forall s\leq t }$.
    First note that $V( \bx(t_0;\bx_0 ) ) = k$ is impossible.
    Indeed, for such a $t_0$ we have $\dot{V}( \bx(t_0;\bx_0) ) < 0$, so 
    $V( \bx(t;\bx_0) ) > V( \bx(t_1;\bx_0) )$ for all $t\in[t_0 - \bar{\epsilon}, t_0]$ for some small enough $\bar{\epsilon}>0$ (by continuity of $\dot{V}$), and therefore $V( \bx(t;\bx_0) ) > k$ for $t\in[t_0 - \bar{\epsilon}, t_0]$, which contradicts the definition of $t_0$.
    Since $V( \bx(t_0;\bx_0 ) ) = k$ is impossible, we necessarily have $V(\bx(t_0;\bx_0)) = \epsilon$. Since 
    $\dot{V}( \bx(t;\bx_0) ) < 0$ for all $t\in[0,T)$, we conclude that $V(\bx(t;\bx_0)) < \epsilon$ for all $t\in(t_0,T)$.
    Since $\epsilon$ is arbitrary, 
    we necessarily have $\lim\limits_{t\to T} \bx(t;\bx_0) = \mathbf{0}_n$.
\end{proof}

Moving beyond the case of one constraint discussed in Remark~\ref{rem:Sontag-formula}, the controller $u^*$ is not available in closed form if multiple constraints are present, and therefore its implementation requires solving a minimization problem at every $\bx$. 
Although $J_{\bp}$ is strictly convex and therefore $u^*(\bx)$ can be found by using off-the-shelf convex solvers, 
this computation needs to be done at every state in continuous time, which is not possible in practice. 
Instead, one often implements a
sample-and-hold version of the controller at a sufficiently high frequency. However, computing the minimizers at such high frequencies becomes computationally burdensome and puts into question the validity of the convergence guarantees. This motivates our ensuing discussion.

\section{Neural Network Approximation of the Universal Formula}\label{sec:neural-network-approximation}

In this section, we present an approach to avoid the computational burden of solving a minimization problem at every state to compute the controller $u^*$ defined in~\eqref{eq:controller}. Our approach is based on computing an approximation of the controller $k^*$ using a NN.
This NN needs to be trained with input-output pairs of the form $(\bp, k^*(\bp))$, with $\bp\in\Pc\subset\real^{N+mN}$.
However, this presents a challenge because the set $\Pc$ is potentially unbounded, which means that the set of possible inputs $\bp$ to the NN is unbounded.
The following result resolves this issue by showing that the values of $J_{\bp}$ can be inferred by only looking at a bounded subset of $\real^{N+mN}$.

\begin{lemma}\longthmtitle{Scaling property of cost function}\label{lem:scaling-property-cost-function}
    Let
    $\tilde{\bp} = (\tilde{A}_1,\hdots,\tilde{A}_N,\tilde{\bB}_1,\hdots,\tilde{\bB}_N)\in\Pc$,
    $\bq = (\tilde{A}_1,\hdots,\tilde{A}_N,\tilde{\bB}_1,\hdots,\tilde{\bB}_N,r)\in\Pc\times[0,1]$ and define the function $\tilde{J}_{\bq}:\Kc_{\tilde{\bp}}\to\real$
    as
    \begin{align*}
        \tilde{J}_{\bq}(\bk) = -\sum_{i=1}^N \frac{\norm{\tilde{\bB}_i}^2 + r\norm{\bk}^2 }{ 2(\tilde{A}_i + \tilde{\bB}_i^\top \bk) }.
    \end{align*}
    Then, $\tilde{J}_{\bq}$ is strictly convex in $\Kc_{\tilde{\bp}}$.  Furthermore, given $\bp = (A_1,\hdots,A_N,\bB_1,\hdots,\bB_N)$, let
    \begin{align*}
        \bq (\bp) \! = \! \Big(\frac{\bp}{M},\frac{1}{M^2} \Big) \! = \! \Big( \frac{A_1}{M},\hdots,\! \frac{A_N}{M},\! \frac{\bB_1}{M},\hdots,\! \frac{\bB_N}{M},\! \frac{1}{M^2} \Big),
    \end{align*}
    where $M = \max\{ |A_1|,\hdots,|A_N|,|\norm{\bB_1},\hdots,\norm{\bB_N}, 1 \}$.     
    Then $J_{\bp}(\bk) = \tilde{J}_{ \bq(\bp) }(\bk)$ for all $\bk\in\Kc_{\bp}$.
   As a consequence, if $\tilde{k}^*:\Pc\times[0,1]\to\real^m$ denotes the function that maps each $\bq\in\Pc\times[0,1]$ to the minimizer of $\tilde{J}_{\bq}$, we have $k^*(\bp) = \tilde{k}^*(\bq(\bp))$ for all $\bp\in\Kc_{\bp}$.
\end{lemma}
\begin{proof}
   The proof that $\tilde{J}_{\bq}$ is strictly convex follows an argument analogous to that of Proposition~\ref{prop:strict-convexity}. Given $\bp = (A_1,\hdots,A_N,\bB_1,\hdots,\bB_N)$, note that 
    \begin{align*}
        \tilde{J}_{\bq (\bp)} (\bk) &= -M \sum_{i=1}^N \frac{ \frac{\norm{\bB_i}^2}{M^2} + \frac{ \norm{\bk}^2 }{ M^2 }  }{ 2( A_i + \bB_i^\top \bk ) } \\
        &=-\frac{1}{M} \sum_{i=1}^N \frac{ \norm{\bB_i}^2 + \norm{\bk}^2  }{ 2( A_i + \bB_i^\top \bk ) } = J_{\bp}(\bk).
    \end{align*}
    Therefore, the minimizers of $\tilde{J}_{\tilde{\bp}}$ and $J_{\bp}$ are the same.
\end{proof}

Lemma~\ref{lem:scaling-property-cost-function} shows that by approximating $\tilde{k}^*$ with values of $\bq$ in $\Tc=\left([-1,1]^N \times \Bc_1(\mathbf{0}_m)^N \right) \cap\Pc \times [0,1]$, we can recover any value of $k^*$.
Therefore, our approach consists in approximating the minimizer of $\tilde{J}_{\bq}$ with a NN for $\bq$ in the set $\Tc$.
Since $\Pc$ is an open set, $\Tc$ is not compact.
Given that the universal approximation theorem of neural networks~\citep{GC:89} is only valid in compact sets,
it does not apply to $\mathcal{T}$.
However, for any compact subset of $\mathcal{T}$ (which we can take to be arbitrarily close to $\mathcal{T}$), the universal approximation theorem ensures that a NN with a sufficiently large number of parameters can approximate $k^*$  arbitrarily well.
In particular, if we are interested in the trajectories evolving on a compact set $\mathcal{X}\subset\Fc$, we only require having a good approximation of $\tilde{k}^*$ in the compact set $\bq( \{ a_i(\mathcal{X}) \}_{i=1}^N, \{ b_i(\mathcal{X}) \}_{i=1}^N )$.
Note also that by choosing the activation functions of the NN to be smooth, the resulting approximation of $k^*$ is $\Cc^{\infty}(\Pc)$, and therefore the approximation of the controller $u^*$ is $\Cc^{l}(\Pc)$.

\begin{remark}\longthmtitle{Applicability for multiple controllers}\label{rem:applicability-multiple-controllers}
    Note that once we have obtained an approximation of $k^*$, we can use it in many different instantiations of the functions $a_1, \hdots, a_N, b_1, \hdots, b_N$, i.e., in a variety of different control problems.
    The NN only depends on the number of constraints $N$ and the dimension of the input $m$, but remarkably, does not depend on $n$, the dimension of the state.
    In fact, given a NN trained with a given input dimension $m$ and number of constraints $N$, all control problems with input dimension at most $m$ and at most $N$ constraints can be solved with the same NN (in case the input dimension is strictly less than $m$ or the number of constraints is strictly less than $N$ one can simply assign the coefficients of the additional inputs or constraints so that they are trivially satisfied).
    In particular, this implies the remarkable fact that all safe stabilization problems for systems with a given input dimension can be solved with the same NN approximation.
    \demo
\end{remark}

\begin{remark}\longthmtitle{Robustness}\label{rem:robustness}
    If one of the inequalities corresponds to a CLF inequality~\eqref{eq:clf-ineq}, and the corresponding CLF is an input-to-state stable (ISS) Lyapunov function for the system~\eqref{eq:control-affine-system}
    (cf.~\cite[Section 3.3]{EDS:08}),
    then if we let $\hat{k}$ be an approximated version of the controller $k^*$ (like the one obtained by approximating it with a NN), 
    and the bound 
    $\normb{ \hat{k}(\bx) - k^*(\bx) } \leq \sigma$ holds for all $\bx\in\mathcal{X}$, with $\mathcal{X}$ some compact set,
    we have
    \begin{align*}
        &\nabla V(\bx)^\top (f(\bx)+g(\bx)\hat{k}(\bx)) \\
        &= \nabla V(\bx)^\top (f(\bx)+g(\bx)k^*(\bx)) + \gamma(\sigma),
    \end{align*}
    for some class $\Kc_{\infty}$ function $\gamma$.
    A similar result holds if the inequality is a CBF inequality~\eqref{eq:cbf-ineq} and the CBF defining it is an input-to-state safe (ISSf) barrier function for the system~\eqref{eq:control-affine-system}~\cite[Theorem 1]{SK-ADA:18}.
    Therefore, CLF and CBF inequalities are robust to approximation errors induced by the NN and therefore the NN-based controller, when applied to safety and stability tasks, makes the closed-loop system close to safe and stable, respectively.
    We would also like to point out that under the assumption that the nominal vector field $f(\bx) + g(\bx)k^*(\bx)$ is Lipschitz, the difference between the trajectories of the closed-loop under $k^*$ and the closed-loop under its NN approximation $\hat{k}$ can be upper bounded in terms of the magnitude of the approximation error $\sigma$~\cite[Theorem 3.4]{HK:02}. In the case where the nominal vector field $f(\bx) + g(\bx)k^*(\bx)$ is contracting, even stronger error bounds can be derived~\cite[Corollary 3.17]{FB-CTDS}.
    \demo
\end{remark}

\begin{remark}\longthmtitle{Hard-constrained NN}\label{rem:hard-constrained-nn}
    We should also point out that the recently introduced HardNet~\citep{YM-NA:25} and $\Pi$-net~\cite{PDG-AT-ECB-RD-JL:25} methods show that by including a projection step at the output layer of the NN, the NN predictions can be guaranteed to satisfy a set of affine constraints.
    Hence, exact stability and safety guarantees (as opposed to ISS/ISSf guarantees, cf. Remark~\ref{rem:robustness}) can be achieved by utilizing these methods.
    Although the universal approximation guarantees are retained, the projection operation in the last layer makes the controller obtained from it locally Lipschitz but not $\Cc^{l}(\Fc)$, for $l\geq 1$. Hence, the existing NN approximation methods under hard constraints impose a trade-off between smoothness and constraint satisfaction. In general, the problem of devising NN architectures that produce smooth approximations, have universal approximation guarantees, and satisfy hard constraints is an interesting line of future research, highly relevant to our work here.
    \demo
\end{remark}

\begin{remark}\longthmtitle{Computational Complexity}\label{rem:computational-complexity}
    Here we discuss how our approach scales with respect to the control dimension $m$ and the number of constraints $N$.
    As shown in~\cite{DY:17}, the number of NN parameters needed to provide an approximation of order $\epsilon>0$ for the function $k^*$ (which has dimension $N(m+1)+1$ is $\mathcal{O}(\epsilon^{-Nm})$).
    In terms of inference time, the complexity of executing a NN-based controller is proportional to the number of parameters of the NN, but is much faster in practice thanks to GPU parallelization. However, as mentioned in Remark~\ref{rem:applicability-multiple-controllers}, the advantage of the NN-based controller is that once it has been obtained for a given $m$ and $N$, all tasks with up to $m$ inputs and $N$ constraints can be solved with the same NN.
    In comparison, the CLF-CBF-QP~\cite{ADA-XX-JWG-PT:17}, requires computing the solution of the QP in real time.
    As shown in~\cite{SJW:97}, 
    the complexity (measured as the worst-case number of floating point operations) of solving the QP using an interior-point method is of order $\mathcal{O}( (m+N)^3 \sqrt{m} )$.
    In practice, since the number of NN parameters is smaller than $\mathcal{O}((m+N)^3 \sqrt{m})$, this makes the NN-based controller computationally cheaper.
    \demo
\end{remark}

\section{Simulations}\label{sec:simulations}

In this section we illustrate our approach in different simulation examples.
We train NNs to approximate $k^*$ as detailed in Section~\ref{sec:neural-network-approximation}~\footnote{
The code and additional details of our simulation can be found at \url{https://github.com/Bluernat/Neural-Network-based-Universal-Formulas-for-Control}.}.

We focus on the cases $m=N=2$ and $m=N=10$, and we train a different NN for each case.

For the case $m=N=2$, we use a feedforward NN with 4 layers (with input dimension $7$ and output dimension $2$), each with 64 neurons, and we use the Sigmoid Linear Unit (SiLU)
activation function $\phi:\real\to\real$ given by $\phi(s) = \frac{s}{1+\exp\{ -s \} }$~\citep{SE-EU-KD:18}, which is a smooth approximation of the ReLU activation function.
The NN is trained with the Adam optimizer with learning rate $3 \times 10^{-3}$ for $2000$ epochs using a dataset with $25113$ data points. 
Each data point $\bq$
consists of a uniformly randomly sampled point in $\left( [-1,1]\times\Bc_1(\mathbf{0}_2)\times[-1, 1]\times\Bc_1(\mathbf{0}_2)\times[0,1] \right) \cap \Pc \subset \real^7$ 
(to do so, we first uniformly sample a point from $[-1,1]\times\Bc_1(\mathbf{0}_2)\times[-1, 1]\times\Bc_1(\mathbf{0}_2)\times[0,1]$ and then check whether the corresponding 2 inequalities are feasible with the QP solver of the \texttt{cvxpy}
library in Python~\citep{SD-SB:16}), 
along with the corresponding minimizer of the function $\tilde{J}_{\bq}$.
We compute such minimizer 
up to a tolerance of $10^{-6}$ by numerically integrating the gradient flow 
$\dot{\bk} = -\nabla \tilde{J}_{\bq}(\bk)$ 
using Python's \texttt{solve\_ivp} function in the \texttt{SciPy} library~\citep{PV-RG-TEO:20}.
We also train the same NN architecture with the HardNet method from~\citep{YM-NA:25}, which includes a projection step at the output layer and guarantees that the NN predictions satisfy the two affine constraints (cf. Remark~\ref{rem:hard-constrained-nn}).

For the case $m=N=10$, we use a residual NN with 6 layers (with input dimension $111$ and output dimension $10$), each with $256$ neurons (except for the last one, that has 128), and also SiLU activation function. We follow an analogous process to the case $m=N=2$ to generate the training dataset, and also train the same NN architecture with the HardNet method.
All simulations were run in an Ubuntu PC with GPU P100.

\begin{example}\longthmtitle{Safe stabilization of single-integrator system}\label{ex:safe-stabilization-single-integrator}
    Consider a single integrator in~$\real^n$,
    \begin{align}\label{eq:single-integrator}
        \dot{\bx} = \bu.
    \end{align}
    Suppose that our goal is to design a controller that stabilizes the system to the origin, and stays in the safe set
    $\Cc = \setdef{\bx\in\real^n}{h_i(\bx) = \norm{ \bx-\bc_i }^2 - r_i^2 \geq 0, \ i\in[\bar{n}]}$, where $\bar{n}\in\mathbb{N}$ is the number of obstacles.
    Note that $V(\bx) = \frac{1}{2}\norm{\bx}^2$ is a CLF for~\eqref{eq:single-integrator}. 
    We further take the positive definite function $W$ in~\eqref{eq:clf-ineq} to be $W(\bx) = 0.1 \norm{\bx}^2$.
    First, we consider the case $n = m = 2$ and $\bar{n} = 3$, with $\bc_1 = (0, 2.5)$, $r_1 = 1$, $\bc_2 = (-2,-2)$, $r_2 = 1$, $\bc_3 = (2, -2)$, $r_3 = 1$.
    We let $h(\bx) = h_1(\bx)h_2(\bx)h_3(\bx)$, which can be shown to be a CBF of $\Cc$.
    We consider the extended class $\Kc_{\infty}$ function $\alpha$ in~\eqref{eq:cbf-ineq} to be $\alpha(s) = s$.
    Next, we take $a_1(\bx) = 0.1\norm{\bx}^2$, $b_1(\bx) = \bx$, $a_2(\bx) = -h(\bx)$, and $b_2(\bx) = -\nabla h(\bx)$.
    By using~\citep[Lemma 5.2]{PM-JC:23-csl}, we can show that the inequalities $a_1(\bx) + b_1(\bx)^\top \bu < 0$, $a_2(\bx) + b_2(\bx)^\top \bu < 0$ are simultaneously feasible at all points in $\Cc$ where $\bx$ and $\nabla h(\bx)$ are linearly independent.

    In this setting, we can use the pretrained NN with $m = N = 2$.
    Figure~\ref{fig:ballobstacle} compares the trajectories obtained by directly applying the NN-based controller in closed-loop and using it to warmstart an optimization scheme to compute $u^*$ online.
    We observe that using the NN-based controller induces trajectories that are safe and converge to a small neighborhood around the origin, for the chosen initial conditions.
    Instead, the trajectories that use a controller obtained by numerically computing $u^*$ online with the NN controller as a warmstart asymptotically converge to the origin (instead of a neighborhood of it) and are also safe.
    The same is true for the controllers obtained from the HardNet method, the CLF-CBF QP, or the computation of $u^*$ online with the CLF-CBF QP as a warmstart.

    \begin{figure}[tbh]
    \includegraphics[width = 0.23\textwidth]{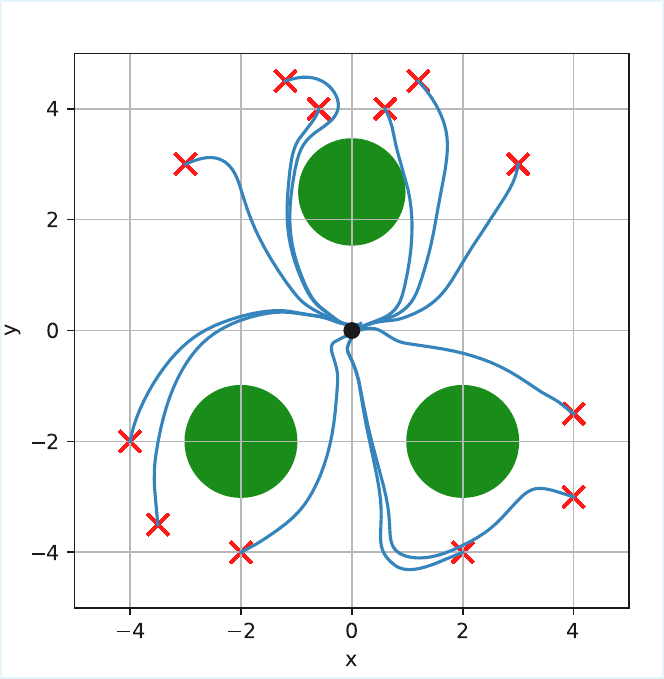}
    \includegraphics[width = 0.23\textwidth]{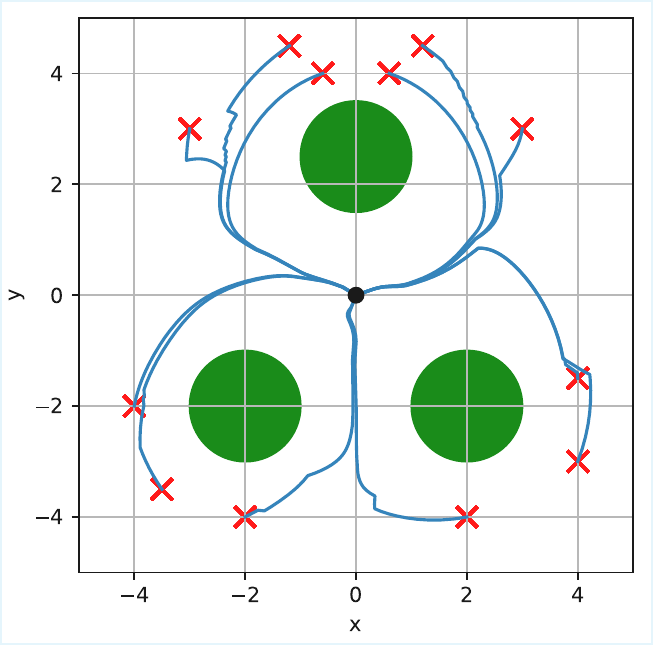}
    \caption{(left) Trajectories of the closed-loop system obtained from the neural network based controller for Example~\ref{ex:safe-stabilization-single-integrator}. (right)
    Trajectories of the closed-loop system obtained from numerically finding the controller $u^*$ online and warmstarting the solver with the NN-based controller for Example~\ref{ex:safe-stabilization-single-integrator}.
    Initial conditions are denoted by red crosses, the origin is the black dot, and the green region denotes the unsafe set.}
    \label{fig:ballobstacle}
    \vspace{-2ex}
    \end{figure}

    Next, we consider the case $n = m = 10$ and $\bar{n} = 9$, with $r_i = 0.8$ for $i\in[9]$ and $\{ \bc_i \}_{i=1}^9$ located at $[-2.5, 2.5]^{10}$. In this case, we consider CBFs $\tilde{h}_i(\bx) = 8(1 - \frac{r_i^2}{\norm{\bx-\bc_i}^2 })$, and consider the extended class $\Kc_{\infty}$ function $\alpha$ in~\eqref{eq:cbf-ineq} to be $\alpha(s) = s$. We further define $a_i(\bx) = -\tilde{h}_i(\bx)$, $b_i(\bx) = -\nabla \tilde{h}_i(\bx)$ for $i\in[9]$ and $a_{10}(\bx) = 0.1\norm{\bx}^2$, $b_{10}(\bx) = \bx$ (although as mentioned in Remark~\ref{rem:applicability-multiple-controllers}, the same trained NN can be used for any other set of obstacles defined through parameters $\{ r_i, \bc_i \}_{i=1}^{9}$
    (as well as any other system with $m=10$, $N=10$)). 
    Note that $V(\bx) = \frac{1}{2}\norm{\bx}^2$ is a CLF and 
    $\tilde{h}_i(\bx) = 8(1 - \frac{r_i^2}{ \norm{\bx-\bc_i}^2 })$
    is a CBF for $i\in[9]$.
    By using an argument analogous to that of~\citep[Lemma 5.2]{PM-JC:23-csl}, the inequalities $\{ a_i(\bx) + b_i(\bx)^\top \bu < 0 \}_{i=1}^{10}$, are simultaneously feasible at all points where the vectors $\{ \bx-\bc_i \}_{i=1}^9$ and $\bx$ are linearly independent.
    Hence, the set of points where such inequalities are infeasible has measure zero.
    
    In order to improve prediction accuracy, we finetune the pretrained NN (with 40000 samples and only adjusting the weights in the last NN layer) with the given system dynamics, CLF and CBFs.
    
    Given that in this case the dimensionality of the system is too high for the trajectories to be visualized, in Figure~\ref{fig:nn-cbf-clf-evolution} we plot the evolution of the minimum value among all CBFs, as well as the evolution of the CLF under the NN-based controller. Even though this controller does not have safety or stability guarantees, the plots show that safety is satisfied in practice, whereas stability holds for a neighborhood of the origin.
    We also note that although $u^*$ is guaranteed to decrease the CLF at every point, its NN approximation can result in a slight increase of the CLF value, specially for small values of the CLF, as seen in Figure~\ref{fig:nn-cbf-clf-evolution} (bottom).
    \begin{figure}
        \centering
        \includegraphics[width=0.99\linewidth]{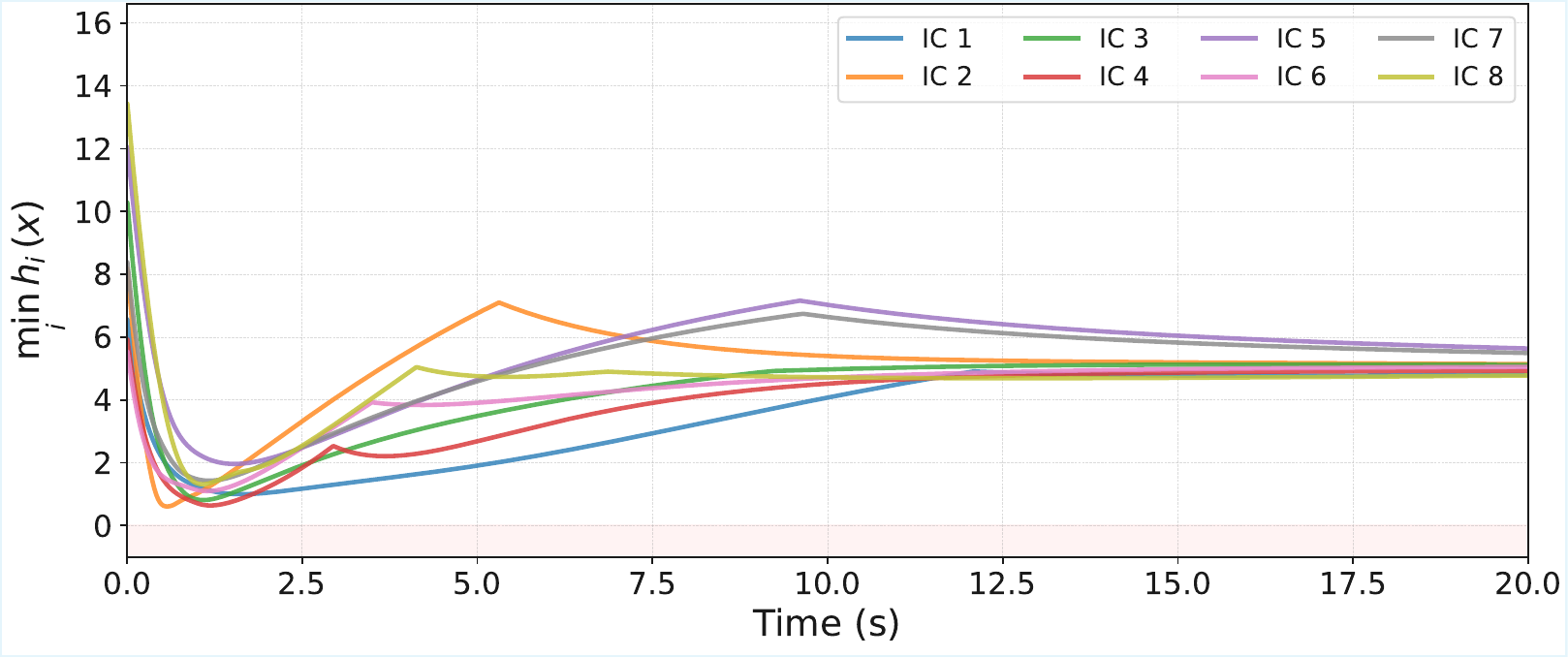}
        \includegraphics[width=0.99\linewidth]{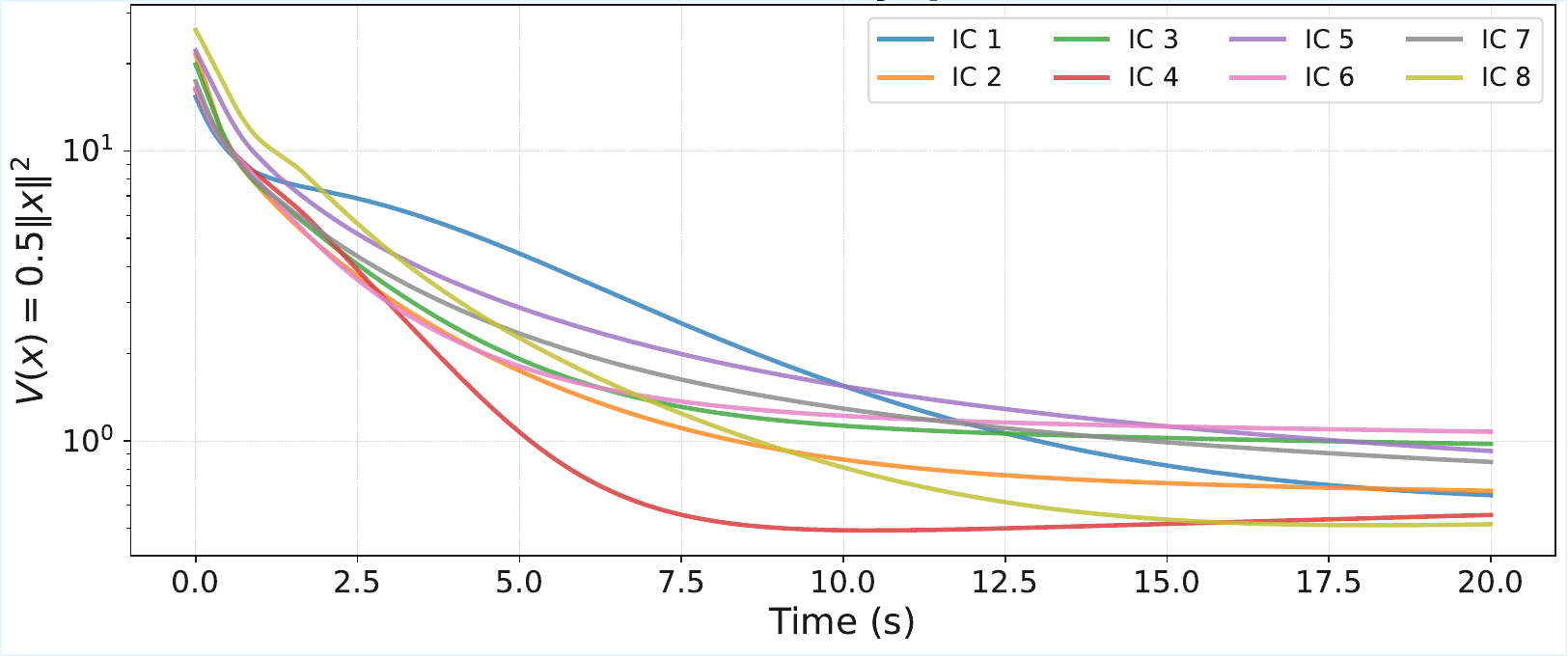}
        
        \caption{ (top) Evolution of $\min\limits_{i\in[9]} h_i(\bx)$ along different trajectories induced by the NN-based controller from different initial conditions (IC). (bottom) Evolution of $V(\bx)$ along trajectories induced by the NN-based controller.}
        \label{fig:nn-cbf-clf-evolution}
    \end{figure}
    On the other hand, the same plots for trajectories obtained with the HardNet-based controller (cf. Figure~\ref{fig:hardnet-cbf-clf-evolution}) show that in addition to safety, the CLF is monotonically decreasing for all times.
    \demo
    
    \begin{figure}
        \centering
        \includegraphics[width=0.99\linewidth]{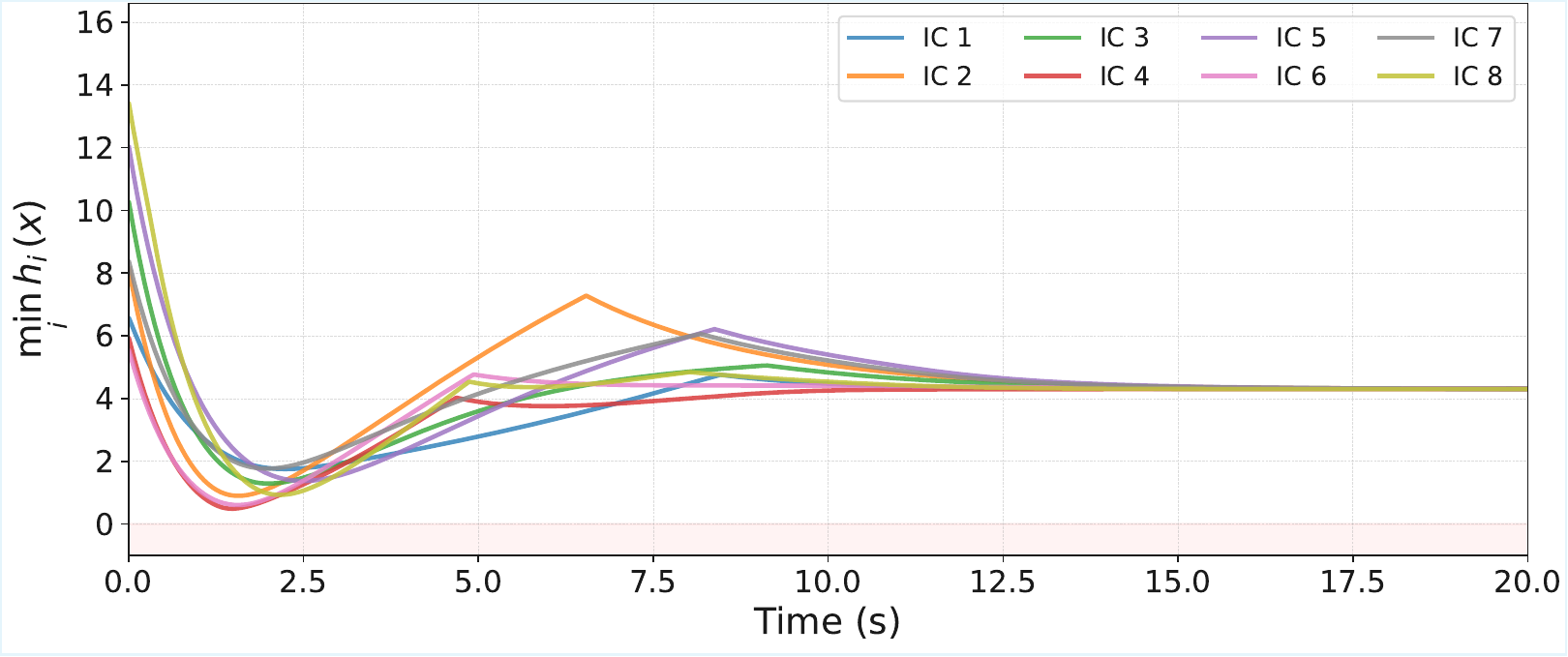}
        \includegraphics[width=0.99\linewidth]{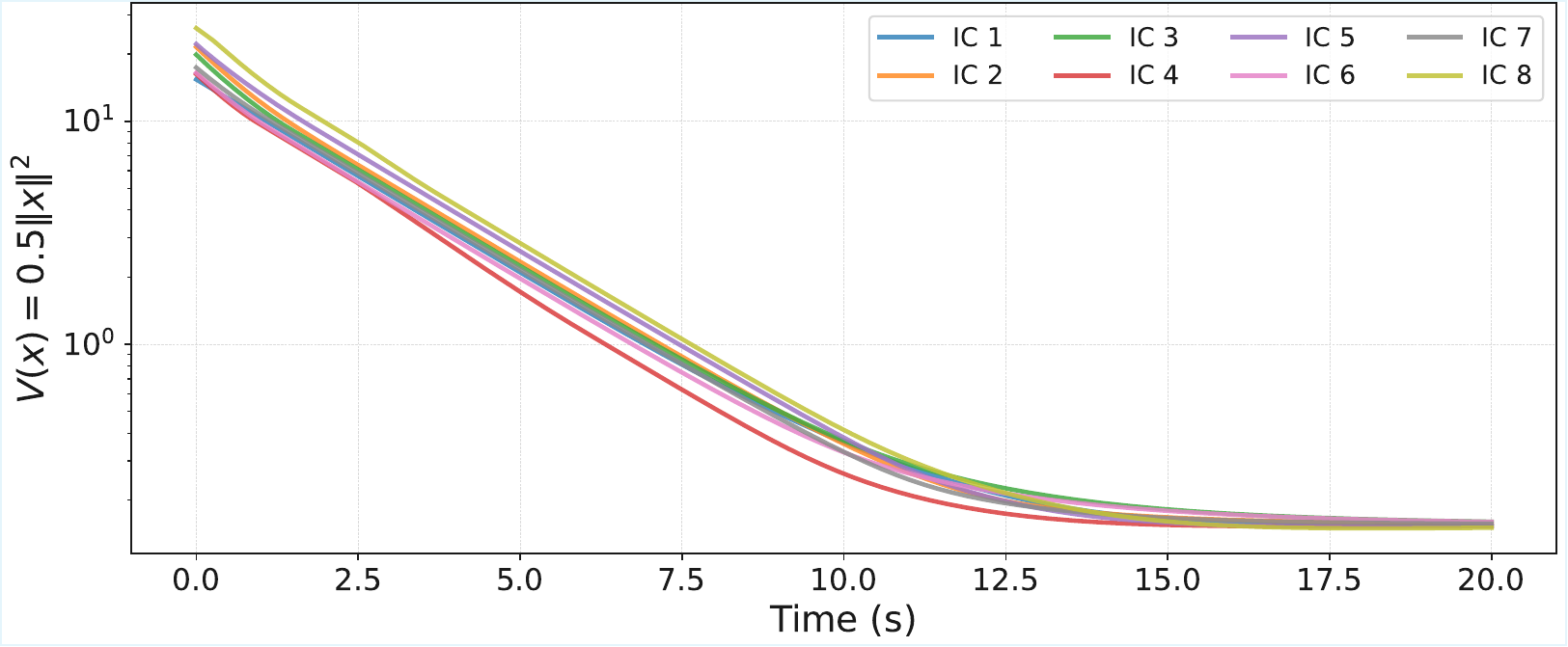}
        
        \caption{ (top) Evolution of $\min\limits_{i\in[9]} h_i(\bx)$ along trajectories induced by the HardNet-based controller from different initial conditions (IC). (bottom) Evolution of $V(\bx)$ along trajectories induced by the HardNet-based controller.}
        \label{fig:hardnet-cbf-clf-evolution}
    \end{figure}
    
\end{example}

\begin{example}\longthmtitle{Safe stabilization of unicycle with drift}\label{ex:safe-stabilization-unicycle-drift}
    Consider the following control system, modeling a unicycle with drift navigating on the plane:
    \begin{align*}
        \dot{x} = v\cos(\theta), \quad 
        \dot{y} = -y + v\sin(\theta), \quad
        \dot{\theta} = \omega,
    \end{align*}
    with $\bu = (v,\omega)$ being the control input.
    Suppose that our goal is to design a controller that stabilizes the system to the origin and stays in the safe set $\Cc = \setdef{(x,y,\theta)\in\real^3}{h(x,y,\theta) = -y + (2x+1)^2 + 1 \geq 0}$ (this example is taken from~\cite[Section VII]{PO-JC:19-cdc}.
    Take the positive definite function $W(x,y) = 0.1(x^2+y^2+\theta^2)$ and $\alpha(s) = s$.
    Now, using the inequalities~\eqref{eq:clf-ineq},~\eqref{eq:cbf-ineq}, we can define
    $b_1(x,y,\theta) = ( x\cos(\theta)+y\sin(\theta),\theta )$
    $a_1(x,y,\theta) = -y^2 + 0.1(x^2+y^2+\theta^2)$,
    $b_2(x,y,\theta) = ( -4(2x+1)\cos(\theta)+\sin(\theta), 0 )$,
    $a_2(x,y,\theta) = -y  -2 h(x,y)$.
    As shown in~\cite[Section VII]{PO-JC:19-cdc}, the inequalities $a_1(x,y,\theta) + b_1(x,y,\theta)^\top \bu < 0 $, 
    $a_2(x,y,\theta) + b_2(x,y,\theta)^\top \bu < 0$ are simultaneously feasible in all of $\Cc\backslash\{ \mathbf{0}_2 \}$, so $\Cc\backslash\{ \mathbf{0}_2 \} \subset \Fc$.
    Since the dimension of the input is $2$ and we have $2$ constraints, we can use the same NN as in Example~\ref{ex:safe-stabilization-single-integrator}, even though the state dimension is different.
    Figure~\ref{fig:unicycle} compares the trajectories obtained by directly applying the NN-based controller in closed-loop and using it to warmstart an optimization scheme to compute $u^*$ online.
    In this case, both controllers induce safe trajectories that asymptotically converge to the origin.
    Table~\ref{tab:execution-times-example1} reports the execution times of various controllers for this example as well.
    \begin{figure}[t]
        \centering
        \includegraphics[width = 0.23\textwidth]{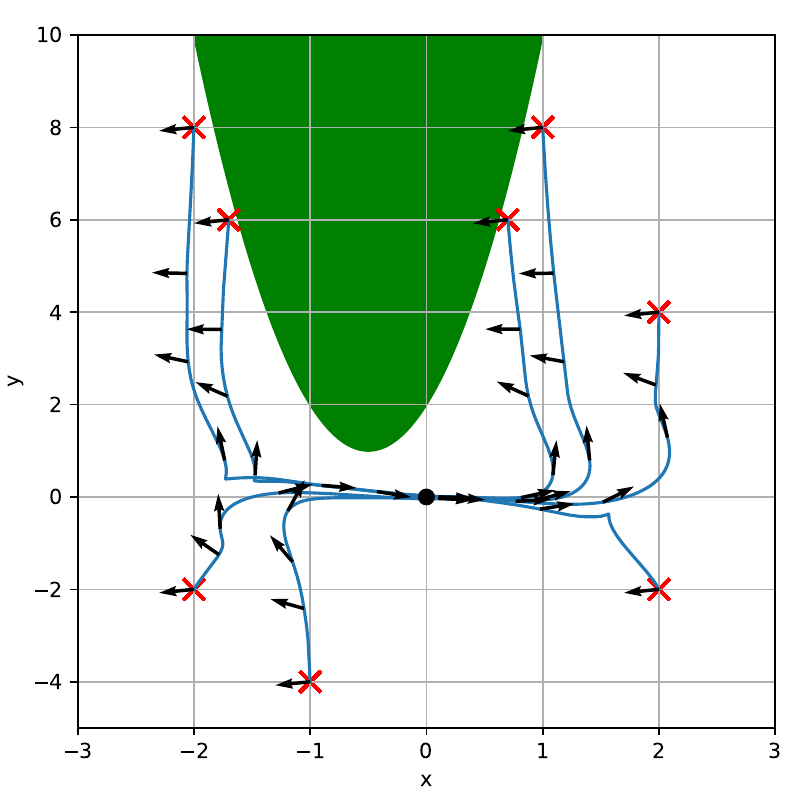}
        \includegraphics[width = 0.23\textwidth]{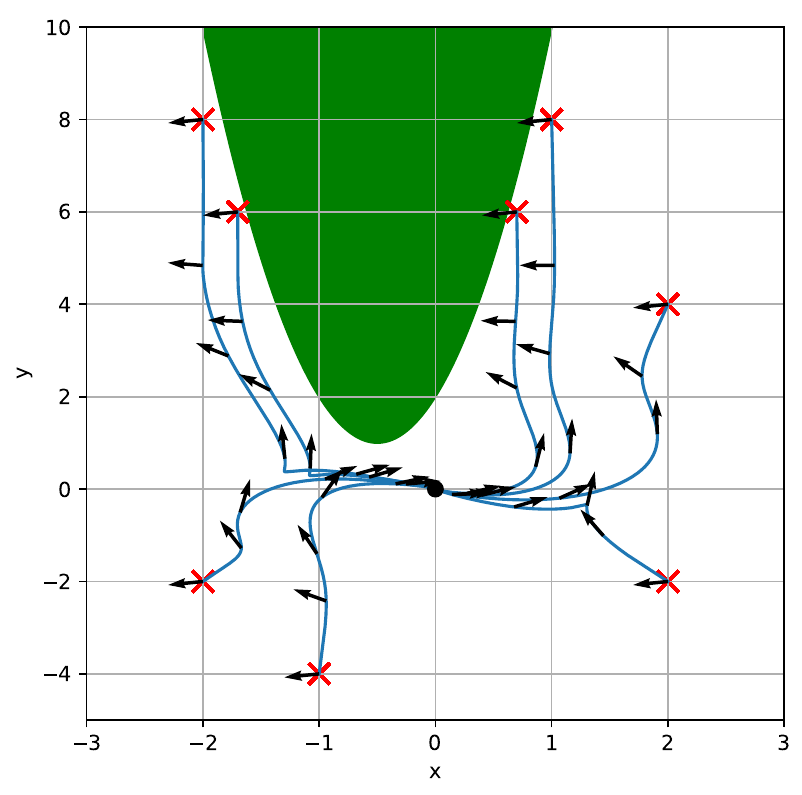}
        \caption{(left) Projection in the $(x,y)$ plane of trajectories of the closed-loop system obtained from the neural network based controller for Example~\ref{ex:safe-stabilization-unicycle-drift}. (right)
        Projection in the $(x,y)$ plane of trajectories of the closed-loop system obtained from numerically finding the controller $u^*$ online and warmstarting the solver with the NN-based controller for Example~\ref{ex:safe-stabilization-unicycle-drift}.
        Initial conditions are denoted by red crosses (and all have an initial orientation $\theta_0 = \pi+0.1$), the origin is the black dot, and the green region denotes the unsafe set. Black arrows indicate the orientation of the unicycle (i.e., the $\theta$ variable) at that point of the trajectory. Observe that the velocity $v$ could be negative, so that at points near the right of the target, the vehicle is "backing up". We also note that due to the presence of a drift term in the $y$ dynamics, the orientation of the unicycle might not be tangent to the trajectory.}
        \label{fig:unicycle}
        \vspace{-3ex}
    \end{figure}
    \demo
\end{example}

\begin{table*}[htb]
    \centering
    {\small
    \begin{tabular}{| c | c | c | c | c | c |}
    \hline
    & 
    NN 
    &
    HardNet
    & 
    \begin{tabular}{@{}c@{}} CLF-CBF QP \\ 
    controller
    \end{tabular}
    & 
    \begin{tabular}{@{}c@{}} $u^*$ with \\ CLF-CBF QP warmstart \end{tabular}
    &
    \begin{tabular}{@{}c@{}} $u^*$ with \\ NN warmstart \end{tabular}
    \\
    \hline
    \begin{tabular}{@{}c@{}} Time (ms) \\ Ex.~\ref{ex:safe-stabilization-single-integrator} $(m=N=2)$
    \end{tabular}
    & $0.053\pm 0.01$ 
    & $0.2 \pm 0.08$
    & $1.72 \pm 0.18$ & $6.9 \pm 2.0$ & $4.7 \pm 1.2$ \\
    \hline
    \begin{tabular}{@{}c@{}} Time (ms) \\ Ex.~\ref{ex:safe-stabilization-single-integrator} $(m=N=10)$
    \end{tabular}
    & $0.680 \pm 0.08$ 
    & $0.823 \pm 0.09$
    & $2.38 \pm 0.37$ & $20.38 \pm 8.35$ & $18.07 \pm 7.16$ \\
    \hline
    \begin{tabular}{@{}c@{}} Time (ms) \\ Ex.~\ref{ex:safe-stabilization-unicycle-drift} 
    \end{tabular}
    & $0.05\pm 0.01$ & $0.2 \pm 0.04$ & $1.72 \pm 0.18$ & $2.9\pm 1.9$ & $1.3 \pm 0.9$ \\
    \hline
    \end{tabular}
    }
    \vspace*{1ex}
    \caption{Average execution times ($\pm$ standard deviation) in milliseconds over the trajectories in Figure~\ref{fig:ballobstacle} for different controller implementations.
    The first column refers to the controller obtained directly as the prediction of the NN.
    The second column refers to the controller obtained as the prediction of the NN when trained with the HardNet method from~\citep{YM-NA:25}.
    The third column refers to the controller obtained by solving the CLF-CBF QP from~\citep{ADA-XX-JWG-PT:17} (using the \texttt{cvxpy} library in Python~\citep{SD-SB:16}).
    The fourth (resp. fifth) column refers to the controller obtained by finding the minimum of the function $\tilde{J}_{\bq}$ using a numerical solver (the \texttt{solve\_ivp} function in Python's \texttt{SciPy} library) 
    and warmstarting the solver with the CLF-CBF QP (resp. the NN prediction).}\label{tab:execution-times-example1}
\end{table*}

Table~\ref{tab:execution-times-example1} reports the execution times of various controllers.
The NN-based controller is the fastest, but does not possess safety and stability guarantees, and is only  approximately optimal (i.e., the NN approximation introduces a suboptimality gap).
The HardNet-based controller is one order of magnitude slower but has safety and stability guarantees by construction, remaining only approximately optimal.
The CLF-CBF QP controller and the controllers obtained by computing $u^*$ online are 
significantly
slower than the NN-based controller,
but are safe, stable, and optimal (in the sense that they minimize the QP cost or $J_{\bp}$, respectively) by construction up to numerical errors.
It is also worth noting that warmstarting the numerical solver to optimize $\tilde{J}_{\bq}$ with the NN prediction also improves its computational time compared to warmstarting it with the CLF-CBF QP.

\section{Conclusions, Limitations, and Future Work}

We have studied the problem of designing a controller that satisfies an arbitrary number of affine inequalities at every point in the state space, which arise when enforcing safety, stability, and input constraints.
We have provided a novel universal formula for controllers satisfying such affine inequalities. The control input is given at every state as the minimizer of a strictly convex function.
To avoid the computation of such minimizer in real time, we have introduced a method based on NN to approximate it.
Remarkably, this NN is universal in the sense that it can be used for any control task with input dimension and number of constraints less than some fixed value, and can be trained with data from just a small subset of the state space.
We have shown the performance of the controller and its NN approximation in various simulation examples.
Future work will 
test our proposed approach to systems with high-dimensional states and inputs, or a high number of constraints. We also plan to extend the proposed methodology to non-affine constraints.

\section{Acknowledgements}

The authors would like to thank Bernat Cots for his valuable help in improving the implementation of the code for generating the training dataset as well as training the NN in the simulations of Section~\ref{sec:simulations}.

\bibliographystyle{unsrtnat}
\bibliography{bib/alias,bib/JC,bib/Main-add,bib/Main}

\begin{thebibliography}{44}
\providecommand{\natexlab}[1]{#1}
\providecommand{\url}[1]{\texttt{#1}}
\expandafter\ifx\csname urlstyle\endcsname\relax
  \providecommand{\doi}[1]{doi: #1}\else
  \providecommand{\doi}{doi: \begingroup \urlstyle{rm}\Url}\fi

\bibitem[Sontag(1998)]{EDS:98}
E.~D. Sontag.
\newblock \emph{Mathematical Control Theory: Deterministic Finite Dimensional
  Systems}, volume~6 of \emph{TAM}.
\newblock Springer, 2 edition, 1998.
\newblock ISBN 0387984895.

\bibitem[Sontag(1983)]{EDS:83}
E.~D. Sontag.
\newblock A {L}yapunov-like characterization of asymptotic controllability.
\newblock \emph{SIAM Journal on Control and Optimization}, 21:\penalty0
  462--471, 1983.

\bibitem[Freeman and Kototovic(1996)]{RAF-PVK:96a}
R.~A. Freeman and P.~V. Kototovic.
\newblock \emph{Robust Nonlinear Control Design: State-space and Lyapunov
  Techniques}.
\newblock Birkhauser Boston Inc., Cambridge, MA, USA, 1996.

\bibitem[Sontag(1989)]{EDS:89a}
E.~D. Sontag.
\newblock A universal construction of {A}rtstein's theorem on nonlinear
  stabilization.
\newblock \emph{Systems \& Control Letters}, 13\penalty0 (2):\penalty0
  117--123, 1989.

\bibitem[Ames et~al.(2019)Ames, Coogan, Egerstedt, Notomista, Sreenath, and
  Tabuada]{ADA-SC-ME-GN-KS-PT:19}
A.~D. Ames, S.~Coogan, M.~Egerstedt, G.~Notomista, K.~Sreenath, and P.~Tabuada.
\newblock Control barrier functions: theory and applications.
\newblock In \emph{{E}uropean {C}ontrol {C}onference}, pages 3420--3431,
  Naples, Italy, 2019.

\bibitem[Ames et~al.(2017)Ames, Xu, Grizzle, and Tabuada]{ADA-XX-JWG-PT:17}
A.~D. Ames, X.~Xu, J.~W. Grizzle, and P.~Tabuada.
\newblock Control barrier function based quadratic programs for safety critical
  systems.
\newblock \emph{IEEE Transactions on Automatic Control}, 62\penalty0
  (8):\penalty0 3861--3876, 2017.

\bibitem[Wieland and Allg{\"o}wer(2007)]{PW-FA:07}
P.~Wieland and F.~Allg{\"o}wer.
\newblock Constructive safety using control barrier functions.
\newblock \emph{IFAC Proceedings Volumes}, 40\penalty0 (12):\penalty0 462--467,
  2007.

\bibitem[Cohen et~al.(2023)Cohen, Ong, Bahati, and Ames]{MC-PO-GB-ADA:23}
M.~Cohen, P.~Ong, G.~Bahati, and A.~D. Ames.
\newblock Characterizing smooth safety filters via the implicit function
  theorem.
\newblock \emph{IEEE Control Systems Letters}, 7:\penalty0 3890--3895, 2023.

\bibitem[Li et~al.(2024)Li, Sun, Koelewijn, and Weiland]{ML-ZS-PJWK-SW:24}
M.~Li, Z.~Sun, P.~J.~W. Koelewijn, and S.~Weiland.
\newblock A tunable universal formula for safety-critical control.
\newblock \emph{arXiv preprint arXiv:2403.06285}, 2024.

\bibitem[Garg and Panagou(2021)]{KG-DP:21}
K.~Garg and D.~Panagou.
\newblock Robust control barrier and control {L}yapunov functions with
  fixed-time convergence guarantees.
\newblock In \emph{{A}merican {C}ontrol {C}onference}, pages 2292--2297, New
  Orleans, LA, July 2021.

\bibitem[Li and Sun(2023)]{ML-ZS:23-acc}
M.~Li and Z.~Sun.
\newblock A graphical interpretation and universal formula for safe
  stabilization.
\newblock In \emph{{A}merican {C}ontrol {C}onference}, pages 3012--3017, San
  Diego, California, 2023.

\bibitem[Romdlony and Jayawardhana(2016)]{MZR-BJ:16}
M.~Z. Romdlony and B.~Jayawardhana.
\newblock Stabilization with guaranteed safety using control {L}yapunov-barrier
  function.
\newblock \emph{Automatica}, 66:\penalty0 39--47, 2016.

\bibitem[Mestres and Cort\'es(2023)]{PM-JC:23-csl}
P.~Mestres and J.~Cort\'es.
\newblock Optimization-based safe stabilizing feedback with guaranteed region
  of attraction.
\newblock \emph{IEEE Control Systems Letters}, 7:\penalty0 367--372, 2023.

\bibitem[Ong and Cort\'es(2019)]{PO-JC:19-cdc}
P.~Ong and J.~Cort\'es.
\newblock Universal formula for smooth safe stabilization.
\newblock In \emph{{IEEE} Conf.\ on Decision and Control}, pages 2373--2378,
  Nice, France, December 2019.

\bibitem[Molnar and Ames(2023)]{TGM-ADA:23}
T.~G. Molnar and A.~D. Ames.
\newblock Composing control barrier functions for complex safety
  specifications.
\newblock \emph{IEEE Control Systems Letters}, 7:\penalty0 3615--3620, 2023.

\bibitem[Tan and Dimarogonas(2024)]{XT-DVD:24}
X.~Tan and D.~V. Dimarogonas.
\newblock On the undesired equilibria induced by control barrier function based
  quadratic programs.
\newblock \emph{Automatica}, 159:\penalty0 111359, 2024.

\bibitem[Mestres et~al.(2025{\natexlab{a}})Mestres, Mousavi, Ong, Yang, Das,
  Burdick, and Ames]{PM-SSM-PO-LY-ED-JWB-ADA:25}
P.~Mestres, S.~S. Mousavi, P.~Ong, L.~Yang, E.~Das, J.~W. Burdick, and A.~D.
  Ames.
\newblock Explicit control barrier function-based safety filters and their
  resource-aware computation.
\newblock 2025{\natexlab{a}}.
\newblock Available at https://arxiv.org/pdf/2512.10118.

\bibitem[Singletary et~al.(2020)Singletary, Chen, and Ames]{AS-YC-ADA:20}
A.~Singletary, Y.~Chen, and A.~D. Ames.
\newblock Control barrier functions for sampled-data systems with input delays.
\newblock pages 804--809, December 2020.

\bibitem[Taylor et~al.(2021)Taylor, Ong, Cort\'es, and
  Ames]{AJT-PO-JC-AA:21-csl}
A.~J. Taylor, P.~Ong, J.~Cort\'es, and A.~Ames.
\newblock Safety-critical event triggered control via input-to-state safe
  barrier functions.
\newblock \emph{IEEE Control Systems Letters}, 5\penalty0 (3):\penalty0
  749--754, 2021.

\bibitem[Breeden et~al.(2022)Breeden, Garg, and Panagou]{JB-KG-DP:22}
J.~Breeden, K.~Garg, and D.~Panagou.
\newblock Control barrier functions in sampled-data systems.
\newblock \emph{IEEE Control Systems Letters}, 6:\penalty0 367--372, 2022.

\bibitem[Taylor et~al.(2022)Taylor, Ong, Molnar, and Ames]{AJT-PO-TGM-ADA:22}
A.~J. Taylor, P.~Ong, T.~G. Molnar, and A.~D. Ames.
\newblock Safe backstepping with control barrier functions.
\newblock In \emph{{IEEE} Conf.\ on Decision and Control}, pages 5775--5782,
  Canc\'un, Mexico, 2022.
\newblock \doi{10.1109/CDC51059.2022.9992763}.

\bibitem[Morris et~al.(2015)Morris, Powell, and Ames]{BJM-MJP-ADA:15}
B.~J. Morris, M.~J. Powell, and A.~D. Ames.
\newblock Continuity and smoothness properties of nonlinear optimization-based
  feedback controllers.
\newblock In \emph{{IEEE} Conf.\ on Decision and Control}, pages 151--158,
  Osaka, Japan, Dec 2015.

\bibitem[Alyaseen et~al.(2025)Alyaseen, Atanasov, and
  Cort\'es]{MA-NA-JC:25-tac}
M.~Alyaseen, N.~Atanasov, and J.~Cort\'es.
\newblock Continuity and boundedness of minimum-norm {CBF}-safe controllers.
\newblock \emph{IEEE Transactions on Automatic Control}, 70\penalty0
  (6):\penalty0 4148--4154, 2025.

\bibitem[Mestres et~al.(2025{\natexlab{b}})Mestres, Allibhoy, and
  Cort\'es]{PM-AA-JC:25-ejc}
P.~Mestres, A.~Allibhoy, and J.~Cort\'es.
\newblock Regularity properties of optimization-based controllers.
\newblock \emph{European Journal of Control}, 81:\penalty0 101098,
  2025{\natexlab{b}}.

\bibitem[Chen et~al.(2022)Chen, Wang, Atanasov, Kumar, and
  Morari]{SWC-TW-NA-VK-MM:22}
S.~W. Chen, T.~Wang, N.~Atanasov, V.~Kumar, and M.~Morari.
\newblock Large scale model predictive control with neural networks and primal
  active sets.
\newblock \emph{Automatica}, 135:\penalty0 109947, 2022.

\bibitem[Chen et~al.(2018)Chen, Saulnier, Atanasov, Lee, Kumar, and
  Pappas]{SC-KS-NA-DDL-VK-GP:18}
S.~Chen, K.~Saulnier, N.~Atanasov, D.~D. Lee, V.~Kumar, and G.~Pappas.
\newblock Approximating explicit model predictive control using constrained
  neural networks.
\newblock In \emph{{A}merican {C}ontrol {C}onference}, pages 1520--1527,
  Milwaukee, Wisconsin, {USA}, 2018.

\bibitem[Hertneck et~al.(2018)Hertneck, Köhler, Trimpe, and
  Allgöwer]{MH-JK-ST-FA:18}
M.~Hertneck, J.~Köhler, S.~Trimpe, and F.~Allgöwer.
\newblock Learning an approximate model predictive controller with guarantees.
\newblock \emph{IEEE Control Systems Letters}, 2\penalty0 (3):\penalty0
  543--548, 2018.

\bibitem[Zhang et~al.(2019)Zhang, Bujarbaruah, and Borrelli]{XZ-MB-FB:19}
X.~Zhang, M.~Bujarbaruah, and F.~Borrelli.
\newblock Safe and near-optimal policy learning for model predictive control
  using primal-dual neural networks.
\newblock In \emph{{A}merican {C}ontrol {C}onference}, pages 354--359,
  Philadelphia, Pennsylvania, {USA}, 2019.

\bibitem[Borrelli et~al.(2017)Borrelli, Bemporad, and Morari]{FB-AB-MM:17}
F.~Borrelli, A.~Bemporad, and M.~Morari.
\newblock \emph{Predictive Control for Linear and Hybrid Systems}.
\newblock Cambridge University Press, Cambridge, UK, 2017.

\bibitem[Min and Azizan(2025)]{YM-NA:25}
Y.~Min and N.~Azizan.
\newblock Hard-constrained neural networks with universal approximation
  guarantees.
\newblock \emph{arXiv preprint arXiv:2410.10807}, 2025.

\bibitem[Grontas et~al.(2025)Grontas, Terpin, Balta, D'Andrea, and
  Lygeros]{PDG-AT-ECB-RD-JL:25}
P.~D. Grontas, A.~Terpin, E.~C. Balta, R.~D'Andrea, and J.~Lygeros.
\newblock Pinet: optimizing hard-constrained neural networks with orthogonal
  projection layers.
\newblock \emph{arXiv preprint arXiv:2508.10480}, 2025.

\bibitem[Mestres and Cort\'es(2025)]{PM-JC:25-tac}
P.~Mestres and J.~Cort\'es.
\newblock Converse theorems for certificates of safety and stability.
\newblock \emph{IEEE Transactions on Automatic Control}, 70\penalty0 (12),
  2025.
\newblock To appear. Available at \url{https://arxiv.org/abs/2406.14823}.

\bibitem[Artstein(1983)]{ZA:83}
Z.~Artstein.
\newblock Stabilization with relaxed controls.
\newblock \emph{Nonlinear Analysis}, 7\penalty0 (11):\penalty0 1163--1173,
  1983.

\bibitem[Dontchev and Rockafellar(2014)]{ALD-RTR:14}
A.~L. Dontchev and R.~T. Rockafellar.
\newblock \emph{{Implicit {F}unctions and {S}olution {M}appings: {A} {V}iew
  from {V}ariational {A}nalysis; 2nd ed.}}
\newblock Springer, New York, NY, 2014.

\bibitem[Khalil(2002)]{HK:02}
H.~Khalil.
\newblock \emph{Nonlinear Systems, 3rd ed.}
\newblock Prentice Hall, Englewood Cliffs, NJ, 2002.

\bibitem[Cybenko(1989)]{GC:89}
G.~Cybenko.
\newblock Dynamic load balancing for distributed memory multiprocessors.
\newblock \emph{Journal of Parallel and Distributed Computing}, 7\penalty0
  (2):\penalty0 279--301, 1989.

\bibitem[Sontag(2008)]{EDS:08}
E.~D. Sontag.
\newblock Input to state stability: Basic concepts and results.
\newblock \emph{Nonlinear and Optimal Control Theory}, 1932:\penalty0 163--220,
  2008.

\bibitem[Kolathaya and Ames(2018)]{SK-ADA:18}
S.~Kolathaya and A.~D. Ames.
\newblock Input-to-state safety with control barrier functions.
\newblock \emph{IEEE Control Systems Letters}, 3\penalty0 (1):\penalty0
  108--113, 2018.

\bibitem[Bullo(2026)]{FB-CTDS}
F.~Bullo.
\newblock \emph{Contraction Theory for Dynamical Systems}.
\newblock Kindle Direct Publishing, {1.3} edition, 2026.
\newblock ISBN 979-8836646806.
\newblock URL \url{https://fbullo.github.io/ctds}.

\bibitem[Yarotsky(2017)]{DY:17}
D.~Yarotsky.
\newblock Error bounds for approximations with deep relu networks.
\newblock \emph{Neural {N}etworks}, 94:\penalty0 103--114, 2017.

\bibitem[Wright(1997)]{SJW:97}
S.~J. Wright.
\newblock \emph{Primal-{D}ual {I}nterior {P}oint {M}ethods}.
\newblock Society for {I}ndustrial and {A}pplied {M}athematics, 1997.

\bibitem[Elfwing et~al.(2018)Elfwing, Uchibe, and Doya]{SE-EU-KD:18}
S.~Elfwing, E.~Uchibe, and K.~Doya.
\newblock Sigmoid-weighted linear units for neural network function
  approximation in reinforcement learning.
\newblock \emph{Neural {N}etworks}, 107:\penalty0 3--11, 2018.

\bibitem[Diamond and Boyd(2016)]{SD-SB:16}
S.~Diamond and S.~Boyd.
\newblock Cvxpy: A python-embedded modeling language for convex optimization.
\newblock \emph{Journal of Machine Learning Research}, 17\penalty0
  (83):\penalty0 1--5, 2016.

\bibitem[Virtanen et~al.(2020)Virtanen, Gommers, Oliphant,
  et~al.]{PV-RG-TEO:20}
P.~Virtanen, R.~Gommers, T.~E. Oliphant, et~al.
\newblock {{SciPy} 1.0: Fundamental Algorithms for Scientific Computing in
  Python}.
\newblock \emph{Nature Methods}, 17:\penalty0 261--272, 2020.

\end{thebibliography}

\end{document}